\documentclass[a4paper, final]{article}
\usepackage{amsfonts, amssymb, amsthm, amsmath}
\usepackage[utf8]{inputenc}
\usepackage{graphicx}
\usepackage[vlined,algoruled]{algorithm2e}
\usepackage{booktabs, multirow}
\usepackage{rotating} 
\usepackage{hyperref}

\usepackage{fontawesome}

\DeclareMathOperator{\Norm}{\mathcal{N}}
\DeclareMathOperator{\avg}{avg}

\newcommand{\R}{\mathbb{R}}
\newcommand{\N}{\mathbb{N}}

\newtheorem{theorem}{Theorem}[section]
\newtheorem{lemma}[theorem]{Lemma}
\newtheorem{proposition}[theorem]{Proposition}
\newtheorem{assumption}[theorem]{Assumption}
\theoremstyle{definition}\newtheorem{defn}[theorem]{Definition}
\newcommand{\raff}{\texttt{RAFF}}
\newcommand{\lmed}{\texttt{LMED}}
\newcommand{\lmlovo}{\texttt{LM-LOVO}}

\title{A robust method based on LOVO functions for solving least
  squares problems\footnote{This project was supported by Fundação
    Araucária, proc. number 002/2017 -- 47223}} %
\author{ %
  \href{https://orcid.org/0000-0001-9718-6486}{\includegraphics[scale=0.6]{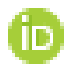}}
  E. V. Castelani\footnote{Department of Mathematics, State University
    of Maringá, Paraná,
    Brazil}\addtocounter{footnote}{-1}\addtocounter{Hfootnote}{-1} %
  \and R. Lopes\footnotemark
  \addtocounter{footnote}{-1}\addtocounter{Hfootnote}{-1} %
  \and
  \href{https://orcid.org/0000-0002-7790-6703}{\includegraphics[scale=0.6]{figure1.eps}}
  W. V. I. Shirabayashi\footnotemark
  \addtocounter{footnote}{-1}\addtocounter{Hfootnote}{-1}%
  \and
  \href{https://orcid.org/0000-0003-4963-0946}{\includegraphics[scale=0.6]{figure1.eps}}
  F. N. C. Sobral \footnotemark\ \footnote{{\faEnvelopeO}~Corresponding author,
    \texttt{fncsobral@uem.br}} %
}

\begin{document}
\maketitle
	
\begin{abstract}
  The robust adjustment of nonlinear models to data is considered in
  this paper. When data comes from real experiments, it is possible
  that measurement errors cause the appearance of discrepant values,
  which should be ignored when adjusting models to them. This work
  presents a Lower Order-value Optimization (LOVO) version of the
  Levenberg-Marquardt algorithm, which is well suited to deal with
  outliers in fitting problems. A general algorithm is presented and
  convergence to stationary points is demonstrated. Numerical results
  show that the algorithm is successfully able to detect and ignore
  outliers without too many specific parameters. Parallel and
  distributed executions of the algorithm are also possible, allowing
  for the use of larger datasets. Comparison against publicly
  available robust algorithms shows that the present approach is able
  to find better adjustments in well known statistical models.
\end{abstract}


  \section{Introduction}
\label{sec:intro}
In this work we are interested in studying the following problem:
given a dataset $\mathcal{R}=\{(t_i,y_i), i=1,...,r\}$ of points in
$\R^m\times \R$, resulting from some experiment, we want to find a
model $\varphi:\R^{m}\rightarrow \R$ for fitting this dataset free
from influence of possible outliers. In a more precise way, given a model $\varphi(t)$
depending on $n$ parameters ($x\in \R^n$), that is,
$\varphi(t)=\phi(x,t)$, we want to find a set
$\mathcal{P}\subset \mathcal{R}$ with $p$ elements and parameters
$\overline{x}\in \R^{n}$, such that
$\phi(\overline{x},t_i)\approx y_i$,
$\forall (t_i,y_i)\in \mathcal{P}$ (in the least squares sense). The
$r-p$ remaining elements in $\mathcal{R}-\mathcal{P}$ are the possible
outliers.

There are several definitions of what an outlier is. The definition
that best suits the present work concerns to errors in $y_i$, that is,
grotesque errors in evaluation of some measure for a given and
  reasonably precise $t_i$. This is somewhat different from the
geometric interpretation of outliers, in the sense that the point
$(t_i, y_i)$ is (geometrically) very far from the graph of a function
that one wants to find.
Typically in our tests, outliers are present when there are errors
resulting from the measurement of some experiment. As a consequence,
their presence may contaminate the obtained model and, therefore,
deteriorate or limit its use. There are several strategies to handle
the presence of outliers in datasets~\cite{Cook77, Hampel86,
  Rousseeuw1984, shif88}.  In a more recent approach, as highlighted
by \cite{Hodge2004} and references therein, techniques based on
machine learning are exploited in the context of deal with a large
amount of data, lack of models and categorical variables.
 
In order to get a fitting model free from influence of outliers, we
use an approach based on {\it Low Order-Value Optimization} (LOVO)
\cite{Andreani2009} which is defined as follows. Consider
$R_i:\R^n \rightarrow \R$, $i=1,...,r$. Given $x \in \R^n$, we can
sort $\{R_i(x), i=1,...,r\}$ in ascending order:
\begin{equation}\label{eq:lovo_order}
  R_{i_1(x)}(x) \leq R_{i_2(x)}(x)\leq ... \leq R_{i_k(x)}(x) \leq \dots \leq R_{i_r(x)}(x),
\end{equation}
where $i_k(x)$ is the $i_k$-th smallest element in that set, for the
given value of $x$.  Given $0<p\leq r$, the LOVO function is defined by

\begin{equation}\label{eq:lovo_fun}
  S_p(x)=\sum_{k=1}^{p} R_{i_k(x)}(x)
\end{equation}
and the LOVO problem is
\begin{equation}\label{eq:lovo_pro1}
  \min S_p(x).
\end{equation}

Essentially, this problem can be seen as a generalization of nonlinear least squares, as elucidated in \cite{Andreani2009}. To reiterate this affirmation, we can consider $\varphi(t)=\phi(x,t)$ as the model selected for fitting,and define $R_i(x)= \dfrac{1}{2}(F_i(x))^2$, where $F_i(x)=y_i-\phi(x,t_i), i=1,...,r$. Thus, we have the particular LOVO problem 
\begin{equation}\label{eq:lovo_main}
  \min S_p(x)=\min \sum_{k=1}^{p}R_{i_k(x)}(x)= \min \sum_{k=1}^p \frac{1}{2}(F_{i_k(x)}(x))^2.
\end{equation}

Each $R_i$ is a residual function. Consequently, if we assume $p=r$
the LOVO problem is the classical least squares problem. When $p<r$
the parameter $\bar{x}\in \mathbb{R}^n$ that
solves~\eqref{eq:lovo_main} defines a model $\phi(\bar{x},t)$ free
from the influence of the worst $r - p$ deviations. Throughout this
work, $p$ is also know as the number of \emph{trusted} points.

Several applications can be modeled in the LOVO context, as is
illustrated in~\cite{Martinez2007, Andreani2009, Andreani2008,
  Birgin2011, JiangHZ17}. An excellent survey about LOVO problems and
variations is given in \cite{Martinez2012a}. Although it is well known
that LOVO deals with detection of outliers, there is a limitation: the
mandatory definition of the value $p$, which is associated to the
number of possible outliers. This is the main gap that this paper
intends to fill. We present a new method that combines a voting schema
and an adaptation of the Levenberg-Marquardt algorithm in context of
LOVO problems.

Levenberg-Marquardt algorithms can be viewed as a particular case of
trust-region algorithms, using specific models to solve nonlinear
equations. In~\cite{Andreani2008}, a LOVO trust-region algorithm is
presented with global and local convergence properties and an
application to protein alignment problems. Second-order derivatives
were needed in the algorithm for the local convergence analysis. In
least-squares problems, as the objective function has a well known
structure, Levenberg-Marquardt algorithms use a linear model for the
adjustment function, instead of a quadratic model for the general
nonlinear function. This approach eliminates the necessity of using
second-order information for the model. More details of
Levenberg-Marquardt methods can be found in~\cite{Nocedal2006}.

In~\cite{Sreevidya2014}, outlier detection techniques are classified
in 7 groups for problems of data streaming: Statistic-based,
depth-based, deviation-based, distance-based, clustering-based,
sliding-window-based and autoregression-based. In~\cite{Yu2010a},
classification is divided only between geometric and algebraic
algorithms for robust curve and surface fitting. The approach used in
this work is clearly algebraic, strongly based in the fact that the
user knows what kind of model is to be used. Although models are used,
we make no assumption on the distribution of the points, so we do not
fit clearly in any of the types described in~\cite{Sreevidya2014}. We
also make the assumption that the values $t_i$ are given exactly, what
is called as \emph{fixed-regressor model} in~\cite{Nocedal2006}.

This work deals with the robust adjustment of models to data. A new
version of the Levenberg-Marquardt algorithm for LOVO problems is
developed, so the necessity of second-order information is avoided. In
addition, the number of possible outliers is estimated by a voting
schema. The main difference of the proposed voting schema is that it
is based in the values of $p$ which has, by definition, a discrete
domain. In other techniques, such as the Hough
Transform~\cite{hough1962,duda1971use,illingworth1988survey,xu1990new},
continuous intervals of the model's parameters are discretized. Also,
the increase in the number of parameters to adjust does not impact the
proposed voting system. The main improvements of this work can be
stated as follows
   \begin{itemize}
   \item a Levenberg-Marquardt algorithm with global convergence for
     LOVO problems is developed, which avoids the use of second-order
     information such as in~\cite{Andreani2008};
   \item a voting schema based on the values of $p$ is developed,
     whose size does not increase with the size or discretization of
     the parameters of the model;
   \item extensive numerical results are presented, which show the
     behavior of the proposed method and are also freely available for
     download.
   \end{itemize}

This work is organized as follows. In Section~\ref{sec:lm-lovo} we
describe the Levenberg-Marquardt algorithm in the LOVO context and
demonstrate its convergence properties. In Section~\ref{sec:raff} the
voting schema is discussed, which will make the LOVO algorithm
independent of the choice of $p$ and will be the basis of the robust
fitting. Section~\ref{sec:impl} is devoted to the discussion of the
implementation details and comparison against other algorithms for
robust fitting. Finally, in Section~\ref{sec:conclusions} we draw some
conclusions of the presented strategy.


\section{The Levenberg-Marquardt method for LOVO problems}
\label{sec:lm-lovo}

Following \cite{Andreani2009}, let us start this section by pointing
out an alternative definition of LOVO problems for theoretical
purposes. Denoting $\mathcal{C}=\{\mathcal{C}_1,...,\mathcal{C}_q\}$
the set of all combinations of the elements $\{1,2,...,r\}$ taken $p$
at a time, we can define for each $i\in \{1,...,q\} $ the following
functions
\begin{equation}\label{eq:lovo_fun_2}
  f_i(x)=\sum_{k \in \mathcal{C}_i} R_k(x)
\end{equation}
and 
\begin{equation}\label{eq:lovo_fun_3}
  f_{min}(x)=\min \{f_i(x),i=1,...,q\}.
\end{equation}
It is simple to note that $S_p(x)=f_{min}(x)$ which is a useful
notation in our context. Moreover, it is possible to note that $S_p$
is a continuous function if $f_i$ is a continuous function for all
$i=1,...,q$, but, even assuming differentiability over $f_i$, we
cannot guarantee the same for $S_p$. In addition, since
$R_k(x) = \dfrac{1}{2}(F_k(x))^2,k\in \mathcal{C}_i,i=1,...,q$ we can
write
\begin{equation}\label{eq: lovo_fun_F}
  f_{i}(x)=\frac{1}{2}\sum_{k\in \mathcal{C}_i} F_{k}(x)^2=\dfrac{1}{2}\|F_{\mathcal{C}_i}(x)\|^2_2.
\end{equation}
Throughout this work, following~\eqref{eq: lovo_fun_F}, given a set
$\mathcal{C}_i \in \mathcal{C}$, $F_{\mathcal{C}_i}(x): \R^n \to \R^p$
will always refer to the map that takes $x$ to the $p$-sized vector
composed by the functions $F_k(x)$ defined by~\eqref{eq:lovo_main},
for $k \in \mathcal{C}_i$ in any fixed order. Similarly,
$J_{\mathcal{C}_i}(x)$ is defined as the Jacobian of this
map. Additionally, we assume the continuous differentiability for
$F_i$, $i=1,...,r$.
	
	The goal of this section is to define a version of Levenberg-Marquardt method (LM) 
	to solve the specific problem (\ref{eq:lovo_main}), for a given $p$, as well as a result on global convergence. The new version  will be called  by simplicity \lmlovo. It is well known that the Levenberg-Marquardt method proposed in \cite{More1978} is closely related to trust-region methods and our approach is based on it. Consequently, some definitions and remarks are necessary. 

	\begin{defn}
		Given $x \in \mathbb{R}^n$ we define the {\it minimal function set of $f_{min}$ in $x$} by
		$$I_{min}(x)=\{i\in\{1,\dots,q\} \ | \ f_{min}(x)=f_i(x)\}.$$
	\end{defn}

	In order to define a search direction for \lmlovo\ at the current point  $x_k$, we choose an index $i \in I_{min}(x_k)$ and compute the direction defined by the classical Levenberg-Marquardt method using $f_i(x)$, that is, the search direction $d_k \in \mathbb{R}^n$ is defined as the solution of
	\begin{equation*} 
		\min_{d \in \mathbb{R}^n}m_{k,i}(d)=\dfrac{1}{2}\Vert F_{\mathcal{C}_i}+J_{\mathcal{C}_i}(x_k)d\Vert_2^2 +\dfrac{\gamma_k}{2}\Vert d \Vert_2^2,
	\end{equation*}
	where $\gamma_k \in \mathbb{R}^+$ is the {\it damping parameter}. Equivalently, the direction $d$ can be obtained by
	\begin{equation}\label{eq:alt_direc}
		(J_{\mathcal{C}_i}(x_k)^T J_{\mathcal{C}_i}(x_k) +\gamma_k I)d=-\nabla f_i(x_k),
	\end{equation}
	where $\nabla f_i(x_k)=J_{\mathcal{C}_i}(x_k)^T F_{\mathcal{C}_i}(x_k)$ and $I \in \mathbb{R}^{n \times n}$ is the identity matrix. 

	To ensure sufficient decrease in the defined search direction, we can consider a similar strategy of trust-region methods, which involves monitoring the actual decrease (given by $f_{min}$) and the predicted decrease (given by $m_{k,i}$) at direction $d_k$:
	\begin{equation}\label{eq:rho}
          \rho_{k,i}=\dfrac{f_{min}(x_k)-f_{min}(x_k+d_k)}{m_{k,i}(0)-m_{k,i}(d_k)}.
        \end{equation}

        We observe that, since $i \in I_{min}(x_k)$,
          $m_{k, i}$ is also a local model for $f_{min}$ at $x_k$. In
          practice, $f_{min}$ can safely be replaced by $f_i$
          in~\eqref{eq:rho}. We formalize the conceptual algorithm
        \lmlovo \ in the Algorithm \ref{alg:LOVO-LMcla}.

\begin{algorithm}[H]
  \caption{{\lmlovo} -- Levenberg-Marquardt for the LOVO problem.}
  \label{alg:LOVO-LMcla}

  \KwIn{$x_0\in\R^n$, $\lambda_{min}, \epsilon, \lambda_0 \in\R_{+}$,
    $\overline{\lambda}>1, \mu\in (0,1)$ and $p \in \N$} %
  \KwOut{$x_k$}

  Set $k \leftarrow 0$\; %
  \BlankLine

  \lnl{alg:lm:selection} Select $i_k\in I_{min}(x_k)$\;
  $\lambda \leftarrow \lambda_k$\; 
  
  \nl \If{$\| \nabla f_{i_k}(x_k)\|_2<\epsilon$}{ %
    Stop the algorithm, $x_k$ is an approximate solution for the LOVO
    problem\; %
  } %
  
  \lnl{alg:lm:optim} $\gamma_k \leftarrow \lambda\| \nabla f_{i_k}(x_k)\|^2_2$\; Compute
  $d_k$ the solution of the linear system \eqref{eq:alt_direc}\;
  Calculate $\rho_{k,i_k}$ as described in \eqref{eq:rho}\;

  \lnl{alg:lm:descscriteria} 
  \uIf{$\rho_{k,i_k}<\mu$}{
  $\lambda \leftarrow \overline{\lambda}\lambda$\;
  Go back to the Step~\ref{alg:lm:optim}\;
  } %
  \Else{ %
    Go to the Step~\ref{alg:lm:successstep}\;
  }
  
    \lnl{alg:lm:successstep}
    $\lambda_{k+1} \in
    [\max\{\lambda_{min},\lambda/\overline{\lambda}\},\lambda]$\;
    $x_{k+1} \leftarrow x_k + d_k$\;
    $k \leftarrow k+1$ and go back to the Step~\ref{alg:lm:selection} \;
\end{algorithm}

In what follows, we show that Algorithm~\ref{alg:LOVO-LMcla} is well
defined and converges to stationary points of the LOVO problem. We
begin with some basic assumptions on the boundedness of the points
generated by the algorithm and on the smoothness of the involved
functions.

\begin{assumption}\label{hip:lips-cont-lovo}
	The level set $$C(x_0)=\{x\in\R^n \ | \ f_{min}(x)\le f_{min}(x_0)\}$$ is a bounded set of $\R^n$ and the functions $f_i$, $i = 1,\dots,r$, have Lipschitz continuous gradients with Lipschitz constants $L_i>0$ in an open set containing $C(x_0)$.
\end{assumption}

The next proposition is classical in the literature of trust-region
methods and ensures decrease of $m_{k,{i_k}}(.)$ on the Cauchy
direction. It was adapted to the LOVO context.

\begin{proposition}\label{pro3}
Given $x_k\in\mathbb{R}^n$, $\gamma\in\mathbb{R}_+$ and ${i_k}\in\{1,\dots,r\}$, the descent direction obtained from  
$$\widehat{t}=argmin_{t\in\R} \ \ \{m_{k,i_k}(-t\nabla f_{i_k}(x_k)\}$$
and expressed by $d^C(x_k)=-\widehat{t}\nabla f_{i_k}(x_k)\in\mathbb{R}^n$, satisfies
\begin{equation}\label{eq:dec-dir-cauchy}
m_{k,i_k}(0)-m_{k,i_k}(d^C(x_k)) \ge 
\dfrac{\theta\|\nabla f_{i_k}(x_k)\|^2_2}{2(\| J_{\mathcal{C}_i}(x_k)\|^2_2 + \gamma)}.
\end{equation}
\end{proposition}

\begin{proof}
  The proof follows from~\cite[Section 3]{Bergou2018}.
\end{proof}

Since the Cauchy's step is obtained by the constant that minimizes the
model $m_{k,i_k}(.)$ on the direction of the gradient vector, we can
conclude that
\begin{eqnarray}\label{eq:dec-cauchy}
m_{k,i_k}(0)-m_{k,i_k}(d_k) \ge 
\dfrac{\theta\|\nabla f_{i_k}(x)\|^2_2}{2(\|J_{\mathcal{C}_i}({x_k})\|^2_2 + \gamma)},
\end{eqnarray}
since $d_k\in\mathbb{R}^n$ from \eqref{eq:alt_direc} is the global minimizer of $m_{k,i_k}$.

Inspired by \cite{Bergou2018}, we present Lemma \ref{lem2} that shows
that Step~\ref{alg:lm:successstep} is always executed by
Algorithm~\ref{alg:LOVO-LMcla} if $\lambda$ is chosen big enough.

\begin{lemma}\label{lem2}
	Let $x_k\in\mathbb{R}^n$ and $i_k\in I_{min}(x_k)$ be a vector and an index fixed in 
	the Step~\ref{alg:lm:selection} of the Algorithm~\ref{alg:LOVO-LMcla}. Then, the Step~\ref{alg:lm:optim} of the Algorithm~\ref{alg:LOVO-LMcla} will be 
	executed a finite number of times.
\end{lemma}

\begin{proof}
	To achieve this goal, we will show that
\begin{equation*}
\lim_{\lambda \rightarrow \infty} \rho_{k,i_k}=2.
\end{equation*}

    For each $\lambda$ fixed in the Step~\ref{alg:lm:selection} of the Algorithm~\ref{alg:LOVO-LMcla}, 
    we have that
\begin{eqnarray}\label{eq:1prof-lem2}
\left|1-\dfrac{\rho_{k,i_k}}{2}\right| & = & \left|1-\dfrac{f_{min}(x_{k})-f_{min}(x_{k}+d_{k})}{2(m_{k,i_k}(0)-m_{k,i_k}(d_{k})} \right|
\nonumber \\
& = & \left|\dfrac{2m_{k,i_k}(0)-2m_{k,i_k}(d_{k}) - f_{min}(x_{k})+f_{min}(x_{k}+d_{k})}{2(m_{k,i_k}(0)-m_{k,i_k}(d_{k}))} \right|
\nonumber \\
& = & \left|\dfrac{f_{min}(x_{k}+d_{k})
+f_{min}(x_{k})-2m_{k,i_k}(d_{k})}{2(m_{k,i_k}(0)-m_{k,i_k}(d_{k}))} \right|\nonumber\\
\end{eqnarray}

From Taylor series expansion and the Lipschitz continuity of $\nabla f_{i_k}(x_{k})$
\begin{equation}\label{eq:2prof-lem2}
f_{i_k}(x_{k}+d_{k}) \le f_{i_k}(x_{k}) + 
\nabla f_{i_k}(x_{k})^Td_{k}+ \dfrac{L_{i_k}}{2}\|d_{k}\|^2_2.
\end{equation}

By equation \eqref{eq:2prof-lem2} and the definition of $f_{min}$, we obtain
\begin{equation}\label{eq:3prof-lem2}
f_{min}(x_{k}+d_{k}) \le f_{i_k}(x_{k}+d_{k})\stackrel{\tiny{\eqref{eq:2prof-lem2}}}{\le} f_{i_k}(x_{k})
+ \nabla f_{i_k}(x_{k})^Td_{k} + \dfrac{L_{i_k}}{2}\|d_{k}\|^2_2.
\end{equation}

Using \eqref{eq:3prof-lem2} and the definition of  $m_{k,i_k}$ in~ \eqref{eq:1prof-lem2}, we get
\begin{equation}
  \label{eq:4prof-lem2}
  \begin{split}
    \left|1 - \right.& \left.\dfrac{\rho_{k,i_k}}{2}\right| =
    \left|\dfrac{f_{min}(x_{k}+d_{k})+f_{min}(x_{k})-2m_{k,i_k}(d_{k})}{2(m_{k,i_k}(0)-m_{k,i_k}(d_{k}))}
    \right| \\
    & = \left|\dfrac{f_{min}(x_{k}+d_{k}) +
        f_{i_k}(x_{k})}{2(m_{k,i_k}(0)-m_{k,i_k}(d_{k}))} \right.  -
    \left.\dfrac{\|F_{\mathcal{C}_i}(x_{k})+J_{\mathcal{C}_i}(x_{k})d_{k}\|^2_2
        +\gamma_k\|d_{k}\|^2_2}{2(m_{k,i_k}(0)-m_{k,i_k}(d_{k}))}
    \right| \\
    & = \left|\dfrac{f_{min}(x_{k}+d_{k}) - f_{i_k}(x_{k})-2\nabla
        f_{i_k}(x_{k})^T d_{k} }{2(m_{k,i_k}(0)-m_{k,i_k}(d_{k}))} \right. \\
    & \qquad \left.-
      \dfrac{d_{k}^T\left(
          J_{\mathcal{C}_i}(x_{k})^TJ_{\mathcal{C}_i}(x_{k})+\gamma_{k}I\right)d_{k}}
      {2(m_{k,i_k}(0)-m_{k,i_k}(d_{k}))} \right| \\
    & \stackrel{\tiny{\eqref{eq:alt_direc}}}{=}
    \left|\dfrac{f_{min}(x_{k}+d_{k}) - f_{i_k}(x_{k})- \nabla
        f_{i_k}(x_{k})^Td_{k}}
      {2(m_{k,i_k}(0)-m_{k,i_k}(d_{k}))}\right| \\
    & \stackrel{\tiny{\eqref{eq:3prof-lem2}}}{\le}
    \left|\dfrac{L_{i_k}\|d_k\|^2_2}
      {4(m_{k,i_k}(0)-m_{k,i_k}(d_{k}))} \right|
  \end{split}
\end{equation}

From \eqref{eq:alt_direc} and the definition of $\gamma_k$, we note that
\begin{equation}\label{eq:5prof-lem2}
\|d_{k}\|_2  \le  \dfrac{\|\nabla f_{i_k}(x_{k})\|_2}{\sigma_{k}+\gamma_{k}}
\le \dfrac{\|\nabla f_{i_k}(x_{k})\|_2}{\gamma_{k}}  = \dfrac{1}{\|\nabla f_{i_k}(x_k)\|_2\lambda},
\end{equation}
where $\sigma_{k}=\sigma_{min}(J_{\mathcal{C}_i}(x_{k})^T J_{\mathcal{C}_i}(x_{k}))$ and $\sigma_{min}(B)$ represents the smallest eigenvalue of $B$.

Replacing \eqref{eq:5prof-lem2} in \eqref{eq:4prof-lem2}, we obtain
\begin{equation}
  \label{eq:6prof-lem2}
  \begin{split}
    \left|1-\dfrac{\rho_{k,i_k}}{2}\right| & \le
    \left|\dfrac{\dfrac{L_{i_k}}{\|\nabla
          f_{i_k}(x_{k})\|^2_2\lambda^2}}
      {4(m_{k,i_k}(0)-m_{k,i_k}(d_{k}))} \right|
    \stackrel{\eqref{eq:dec-cauchy}}{\le}
    \left|\dfrac{\dfrac{L_{i_k}}{\|\nabla
          f_{i_k}(x_{k})\|^2_2\lambda^2}} {\dfrac{4\theta\|\nabla
          f_{i_k}(x_{k})\|^2_2}{2(\| J_{\mathcal{C}_i}(x_{k})\|^2_2 +
          \gamma_{k})}} \right| \\
    & = \dfrac{(\| J_{\mathcal{C}_i}(x_{k})\|^2_2 + \gamma_{k})
      L_{i_k} }{2\theta\|\nabla f_{i_k}(x_{k})\|^4_2 \lambda^2} \le
    \left( \dfrac{\| J_{\mathcal{C}_i}(x_{k})\|^2_2}{\epsilon^4} +
      \dfrac{1}{\epsilon^2} \right) \dfrac{L_{i_k}}{2 \theta \lambda},
  \end{split}
\end{equation}
where the last inequality comes from the definition of $\gamma_k$ in
Algorithm~\ref{alg:LOVO-LMcla} and assuming that $\lambda \ge 1$,
which can always be enforced.

Using~\eqref{eq:6prof-lem2}, since the Jacobian $J_{\mathcal{C}_i}$ is
bounded in $C(x_0)$, we conclude that
\begin{equation*}
\lim_{\lambda \rightarrow \infty} \left|1-\dfrac{\rho_{k,i_k}}{2}\right| =0,
\end{equation*}
which proves the result.
\end{proof}

Our studies move toward showing convergence results for Algorithm \ref{alg:LOVO-LMcla} to stationary points. At this point we should be aware of the fact that LOVO problems admit two types of stationary condition: {\it weak} and {\it strong}~\cite{Andreani2008}.

\begin{defn}\label{def:w/s stationarity}
	A point $x^*$ is a weakly critical point of \eqref{eq:lovo_pro1} when $x^*$ is a stationary point of $f_i$ for some $i\in I_{min}(x^*)$. A point  $x^*$ is a strongly critical point of \eqref{eq:lovo_pro1} if $x^*$ is a stationary point of $f_i$ for all $i\in I_{min}(x^*)$.
\end{defn}

Although the strongly critical condition is theoretically interesting, in this work we are limited to the proof of weakly critical, since this second type is less expensive to verify in practice and therefore more common to deal with. The global convergence to weakly critical points is given by Theorem \ref{maintheorem}.

\begin{theorem}\label{maintheorem}
	Let $\{x_k\}_{k\in\N}$ be a sequence generated by Algorithm \ref{alg:LOVO-LMcla}. 
	Consider $\mathcal{K}'=\{ k \ | \ i_k=i \}\subset\N$ an infinite subset of indexes for $i\in\{1,\dots,r\}$ and assume that Assumption \ref{hip:lips-cont-lovo} holds. Then, for all $x_0\in \mathbb{R}^n$, we have 
	$$\lim_{k \in \mathcal{K}'} \|\nabla f_i(x_{k})\|_2=0.$$
\end{theorem}
\begin{proof}
	Clearly, there is an index $i$ chosen an infinite number of times by Algorithm \ref{alg:LOVO-LMcla} since $\{1,\dots,r\}$ is a finite set.

Let us suppose by contradiction that, for this index $i$, there exist $\epsilon>0$ and $K \in\N$ such that $\|\nabla f_i(x_{k})\|_2\ge \epsilon$, for all $k\ge K$.

Therefore, for each  $k\ge K$, we obtain by the Step~\ref{alg:lm:successstep} of Algorithm~\ref{alg:LOVO-LMcla} that $\lambda_{k}\le \lambda_{K}$. Additionally, we get
\begin{align}
& & \dfrac{f_{min}(x_k)-f_{min}(x_k+d_k)}{m_{k,i}(0)-m_{k,i}(d_k)} \ge \mu 
\nonumber\\
& \Leftrightarrow & f_{min}(x_k)-f_{min}(x_{k+1}) 
\ge \mu (m_{k,i}(0)-m_{k,i}(d_k))
\nonumber\\
& \Leftrightarrow & f_{min}(x_k)-f_{min}(x_k+d_k) 
\stackrel{\eqref{eq:dec-cauchy}}{\ge} \mu 
\left(\theta\dfrac{\|\nabla f_i(x_k)\|_2^2}{\|J_{\mathcal{C}_i}(x_k)\|^2_2+\gamma_{k}}\right)
\nonumber\\
& \Leftrightarrow & f_{min}(x_k)-f_{min}(x_k+d_k) \ge 
\left(\dfrac{\mu\theta\epsilon^2}{\|J_{\mathcal{C}_i}(x_k)\|^2_2+\epsilon^2\lambda_k}\right)
\nonumber\\
& \Leftrightarrow & f_{min}(x_k)-f_{min}(x_k+d_k) \ge 
\left(\dfrac{\mu\theta\epsilon^2}{\max_{k\ge K}\{\|J_{\mathcal{C}_i}(x_k)\|^2_2\}+\epsilon^2\lambda_{K}}\right)
\nonumber\\
& \Leftrightarrow & f_{min}(x_k+d_k)-f_{min}(x_k) \le -
\left(\dfrac{\mu\theta\epsilon^2}{c}\right),\label{eq:fminineq}
\end{align}
where $c=\max_{k\ge K}\{\|J_{\mathcal{C}_i}(x_k)\|^2_2\}+\epsilon^2\lambda_{K}$.

Since  $f_{min}$ is bounded from below and decreases by the constant value  $\mu\theta\epsilon^2/c$ at every iteration $k\ge K$, we have, using the continuity of $f_{min}$, that the sequence $f_{min}(x_k+d_k)$ converges to the minimizer of  
\begin{align*}
\min_{x\in\R^n} \quad & f_{min}(x) 
\end{align*}
and, therefore, 
\begin{equation}\label{eq:fminlim}
	\lim_{k \in \mathcal{K}'} f_{min}(x_k+d_k)-f_{min}(x_k)=0.
\end{equation}

Using (\ref{eq:fminineq}) and (\ref{eq:fminlim}) we have that
$$\lim_{k \in \mathcal{K}'} \dfrac{\mu\theta\epsilon^2}{
\max_{k\ge K}\{\|J_{\mathcal{C}_i}(x_k)\|_2\}+\epsilon^2\lambda_k}=0.$$
Consequently $\lim_{k \in \mathcal{K}'} \lambda_k= \infty$, which contradicts inequality 
$\lambda_k\le \lambda_K<\infty$. We conclude that there is no such $\epsilon$ and, therefore,
\[
  \lim_{k \in \mathcal{K}'} \|\nabla f_i(x_{k})\|_2 = 0.
\]
\end{proof}

\section{The voting system}
\label{sec:raff}

The main drawback of Algorithm~\ref{alg:LOVO-LMcla} is the need to
know the number $p$ of trusted points, which is used by
$S_p$~\eqref{eq:lovo_fun} (or, equivalently, by $f_{min}$). It is not
usual to know the exact number of trusted points in any experiment.

To overcome this difficulty, an algorithm for testing different values
of $p$ was created, detailed by Algorithm~\ref{alg:raff}. The
main idea of the method is to call Algorithm~\ref{alg:LOVO-LMcla} for
several different values of $p$ and store the obtained solution. The
solutions are then preprocessed, where stationary points that are not
global minimizers of their respective problem are
eliminated. This elimination is based on the fact that, if
$\bar x_{p}$ and $\bar x_{q}$, $p < q$, are solutions for their
respective problems, then $S_{p}(\bar x_{p})$ cannot be greater than
$S_{q}(\bar x_{q})$ if they are both global minimizers. Therefore,
if $S_p(\bar x_p) > S_q(\bar x_q)$, then $\bar x_{p}$ is not
a global minimizer and can be safely eliminated. The last steps
(Steps~\ref{alg:raff:similarity} and~\ref{alg:raff:voting})
compute the similarity between each pair of solutions and obtain the
most similar ones. Element $C_p$ of vector $C$ stores the number of
times that some other solution was considered similar to
$\bar x_p$, in the sense of a tolerance $\epsilon$. The most similar solution
with greatest $p$ is considered the robust
adjustment model for the problem. Algorithm~\ref{alg:raff} is
a proposal of a voting system, where the solution that was not
eliminated by the preprocessing and occurred with highest frequency
(in the similarity sense) is selected.

\begin{algorithm}

  \caption{Voting algorithm for fitting problems}

  \label{alg:raff}

  \KwIn{$x_0\in\R^n$,
    $\epsilon \in \R_{+}$ and
    $0 \le p_{min} < p_{max}$} %

  \nl Define $C \in \R^s = \mathbf{0}$, where $s = p_{max} - p_{min} + 1$

  \lnl{alg:raff:lmlovo} Compute $\bar x_p \in \R^n$ by calling
  Algorithm~\ref{alg:LOVO-LMcla} for the given $p$, for all
  $p \in \{p_{min},p_{min}+1,...,p_{max}\}$

  \lnl{alg:raff:preprocess} Preprocess solutions

  \lnl{alg:raff:similarity} Let $M_{pq}$ be the similarity between solutions $\bar x_p$ and
  $\bar x_q$

  \lnl{alg:raff:voting} \For{$p = p_{min}, \dots, p_{max}$}{ %

    $k\leftarrow 0$

    \For{$q = p_{min}, \dots, p_{max}$}{ %
      
      \If{$M_{pq} < \epsilon$}{ $k\leftarrow k+1$ }

      } %

    $C_p \leftarrow k$
    
  }

  \nl $x^{\star} \leftarrow \bar x_p$, where
  $p = \underset{q = p_{min}, \dots, p_{max}}{\arg\max} \{ C_q \}$

\end{algorithm}

The execution of Algorithm~\ref{alg:raff} can be easily
parallelizable. Each call of Algorithm~\ref{alg:LOVO-LMcla} with a
different value of $p$ can be performed independently at
  Step~\ref{alg:raff:lmlovo}. All the convergence results from
Section~\ref{sec:lm-lovo} remain valid, so Algorithm~\ref{alg:raff} is
well defined. All the specific implementation details of the algorithm
are discussed in Section~\ref{sec:impl}.

\section{Numerical implementation and experiments} \label{sec:impl}

In this Section we discuss the implementation details of
Algorithms~\ref{alg:LOVO-LMcla} and~\ref{alg:raff}. From now on,
Algorithm~\ref{alg:LOVO-LMcla} will be called {\lmlovo} and
Algorithm~\ref{alg:raff} will be called {\raff}. Both algorithms were
implemented in the Julia language, version 1.0.4 and are available in
the official Julia repository. See~\cite{Castelani2019} for
information about the \texttt{RAFF.jl} package installation and
usage.

Algorithm {\lmlovo} is a sequential nonlinear programming algorithm,
which means that only the traditional parallelization techniques can
be applied. Since fitting problems have small dimension and a large
dataset, the main gains would be the parallelization of the objective
function, not the full algorithm. Matrix and vector operations are
also eligible for parallelization.

Following traditional LOVO
implementations~\cite{Andreani2007curvedect}, the choice of index
$i_k \in I_{min}(x_k)$ is performed by simply evaluating functions
$F_i(x_k)$, $i = 1, \dots, r$, sorting them in ascending order and
them dropping the $r - p$ largest values. Any sorting algorithm can be
used, but we used our implementation of the selection sort
algorithm. This choice is interesting, since the computational cost is
linear when the vector is already in ascending order, what is not
unusual if {\lmlovo} is converging and $i_{k + 1} = i_k$, for example.

The convergence theory needs the sufficient decrease parameter
$\rho_{k, i_k}$ to be calculated in order to define step acceptance
and the update of the damping parameter. In practice, {\lmlovo} uses
the simple decrease test at Step~\ref{alg:lm:descscriteria}
\begin{equation*}
  \label{imp:simpdec}
  f_{min}(x_k + d_k) < f_{min},
\end{equation*}
which was shown to work well in practice.

The computation of direction $d_k$ is performed by solving the linear
system~(\ref{eq:alt_direc}) by the Cholesky factorization of
matrix
$J_{\mathcal{C}_{i_k}}(x_k)^T J_{\mathcal{C}_{i_k}}(x_k) + \lambda_k
I$. In the case where Steps~\ref{alg:lm:optim}
and~\ref{alg:lm:descscriteria} are repeated at the same iteration $k$,
the QR factorization is more indicated, since it can be reused when
the iterate $x_k$ remains the same and only the dumping factor is
changed. See~\cite{More1978} for more details about the use of QR
factorizations in the Levenberg-Marquardt algorithm. If there is no
interest in using the QR factorization, then the Cholesky
factorization is recommended.

{\lmlovo} was carefully implemented, since it is used as a subroutine
of {\raff} for solving adjustment problems. A solution $\bar x = x_k$
is declared as successful if
\begin{equation}
  \label{impl:success}
  \| \nabla f_{i_k}(\bar x) \|_2 \le \varepsilon
\end{equation}
for some $i_k \in I_{min}(\bar x)$, where $f_{i_k}$ is given
by~(\ref{eq:lovo_fun_2}). The algorithm stops if the gradient cannot
be computed due to numerical errors or if the limit of 400 iterations
has been reached. We also set $\overline{\lambda} = 2$ as default.

In order to show the behavior of {\lmlovo} we solved the problem of
adjusting some data to the one-dimensional logistic model, widely used
in statistics
\[
  \phi(x,  t) =  x_1 + \frac{x_2}{1 + \exp(- x_3 t + x_4)},
\]
where $x \in \R^4$ represents the parameters of the model and
$t \in \R$ represents the variable of the model. In order to generate
random data for the test, the procedures detailed in
Subsection~\ref{sec:impl:outlier} were used. The produced data is
displayed in Figure~\ref{fig:lmlog}, where $r = 10$, $p = 9$ and the
exact solution was $x^* = (6000, -5000, -0.2, -3.7)$. This example has
only $r - p = 1$ outlier.

\begin{figure}[ht]
  \centering

  \includegraphics[width=0.9\textwidth]{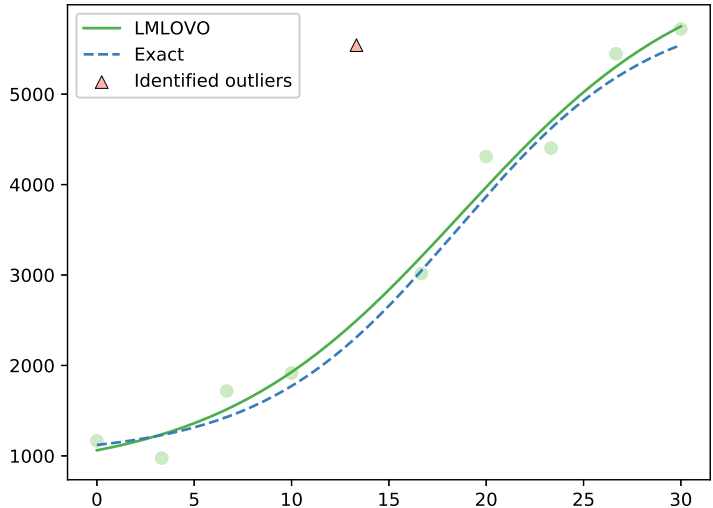}  
  \caption{Test problem simulating an experiment following the
    logistic model. The continuous line represents the adjusted model,
    while the dashed line is the ``exact'' solution. {\lmlovo}
    correctly identifies and ignores the outlier.}
  
\label{fig:lmlog}
\end{figure}

{\lmlovo} was run with its default parameters, using
$x = (0, 0, 0, 0)$ as a starting point and $p = 9$, indicating that
there are 9 points which are trustable for adjusting the model. The
solution found is also shown in Figure~\ref{fig:lmlog}, given by
$\bar x = (795.356, 5749.86, 0.161791, 3.02475)$, as a continuous line,
while the ``exact'' solution is depicted as a dashed line. We observe
that it is not expected the exact solution $x^*$ to be found, since
the points were perturbed. The outlier is correctly identified as the
dark/red triangle.

The example in Figure~\ref{fig:lmlog} has an outlier that is visually
easy to identify, so the correct number of $p = 9$ trusted points was
used. However, that might not be the case, specially if there is an
automated process that needs to perform the adjustments, or if the
model is multi-dimensional. Algorithm {\raff} was implemented to solve
this drawback.

{\raff} was also implemented in the Julia language and is the main
method of the \texttt{RAFF.jl} package~\cite{Castelani2019}. As
already mentioned in Section~\ref{alg:raff}, {\raff} is easily
parallelizable, so serial and parallel/distributed versions are
available, through the \texttt{Distributed.jl} package. The algorithm
(or the user) defines an interval of values of $p$ to test and calls
{\lmlovo} to solve each subproblem
for a given value of $p$. It is known that LOVO
problems have many local minimizers, but we are strongly interested
in global ones. Therefore, the traditional multi-start technique is
applied to generate random starting points. The larger the number of
different starting points, the greater is the chance to find global
minimizers. Also, the computational cost is increased. The
parallel/distributed version of {\raff} solves this drawback,
distributing problems with different values of $p$ among different
cores, processors or even computers.

For the computation of the similarity between solutions $\bar x_p$ and
$\bar x_q$ in Step~\ref{alg:raff:similarity}, the Euclidean norm of
the vector of differences was used
\[
  M_{pq} = \| \bar x_p - \bar x_q \|_2.
\]
For each $p$ in the interval, the best solution $\bar x_p$ obtained
among all the runs of {\lmlovo} for that $p$ is stored. In order to
avoid considering points in the cases where \lmlovo\ has not converged
for some value $p$, we set $M_{ip} = M_{pi} = \infty$ for all
$i = p_{min}, \dots, p_{max}$ in case of failure.

In the preprocessing phase (Step~\ref{alg:raff:preprocess} of \raff)
solutions $\bar x_q$ that clearly are not minimizers are also
eliminated by setting $M_{iq} = M_{qi} = \infty$ for all
$i = p_{min}, \dots, p_{max}$. To detect such points, we check if
$S_q(\bar x_q) > S_p(\bar x_p)$ for some $q < p \le p_{max}$. The idea
is that the less points are considered in the adjustment, the smaller
the residual should be at the global minimizer. The preprocessing
phase also tries to eliminate solution $\bar x_{p_{max}}$. To
  do that, the valid solution $\bar x_p$ with smallest value of
  $S_p(\bar x_p)$, which was not eliminated by the previous strategy,
  is chosen, where $p < p_{max}$. Solution $\bar x_{p_{max}}$ is
eliminated if $S_p(\bar x_p) < S_{p_{max}}(\bar x_{p_{max}})$ and the
number of observed points $(t_i, y_i)$ such that
$|y_i - \phi(\bar x_p, t_i)| < |y_i - \phi(\bar x_{p_{max}}, t_i)|$,
for $i = 1, \dots, r$, is greater or equal than $r / 2$.

The last implementation detail of {\raff} that needs to be addressed
is the choice of $\epsilon$. Although this value can be provided by
the user, we found very hard to select a number that resulted in a
correct adjustment. Very small or very large values of $\epsilon$,
result in the selection of $\bar x_{p_{max}}$ as the solution, since
each solution will be similar to itself or similar to every
  solution, and we always select the largest $p$ in such cases. To
solve this issue, the following calculation has been observed to work
well in practice
\begin{equation}
  \label{impl:tvs}
  \epsilon = \min(M) + \avg(M) / (1 + p_{max}^{1/2}),
\end{equation}
where $M$ is the similarity matrix and function $\avg$ computes the
average similarity by considering only the lower triangular part of
$M$ and ignoring $\infty$ values (which represent eliminated
solutions). If there is no convergence for any value of
$p \in [p_{min}, p_{max}]$, then $\bar x_{p_{max}}$ is returned,
regardless if it has successfully converged or not.

\subsection{Experiments for outlier detection and robust fitting}
\label{sec:impl:outlier}

In the first set of tests, we verified the ability and efficiency of
{\raff} to detect outliers for well known statistical and
mathematical models:
\begin{itemize}
\item Linear model: $\phi(x, t) = x_1 t + x_2$
\item Cubic model: $\phi(x, t) = x_1 t^3 + x_2 t^2 + x_3 t + x_4$
\item Exponential model:
  $\phi(x, t) = x_1 + x_2 \exp( - x_3 t )$
\item Logistic model:
  $\phi(x, t) = x_1 + \frac{x_2}{1 + \exp(- x_3 t + x_4)}$
\end{itemize}
The large number of parameters to be adjusted increases the difficulty
of the problem, since the number of local minima also increases. For
these tests, we followed some ideas described
in~\cite{Motulsky2006a}. For each model, we created 1000 random
generated problems having: 10 points and 1 outlier, 10 points and 2
outliers, 100 points and 1 outlier, and 100 points and 10
outliers. For each combination, we also tested the effect of the
multistart strategy using: 1, 10, 100 and 1000 random starting points.

The procedure for generating each random instance is described as
follows. It is also part of the \texttt{RAFF.jl}
package~\cite{Castelani2019}. Let $x^*$ be the exact solution for this
fitting problem. First, $r$ uniformly spaced values for $t_i$ are
selected in the interval $[1, 30]$. Then, a set
$O \subset \{1, \dots, r\}$ with $r - p$ elements values is randomly
selected to be the set of outliers. For all $i = 1, \dots, r$ a
perturbed value is computed, simulating the results from an
experiment. Therefore, we set $y_i = \phi(x^*, t_i) + \xi_i$,
where $\xi_i \sim \Norm(0, 200)$, if $i \not \in O$ and, otherwise,
$y_i = \phi(x^*, t_i) + 7 s \xi'_i \xi_i$, where
$\xi_i \sim \Norm(0, 200)$, $\xi'_i$ a uniform random number between 1
and 2 and $s \in \{-1, 1\}$ is randomly selected at the beginning of
this process (so all outliers are in the ``same side'' of the
curve). The exact solutions used to generate the instances are given
in Table~\ref{tab:exactsol}. The example illustrated in
Figure~\ref{fig:lmlog} was also generated by this procedure.

\begin{table}[ht]
  \centering
  \begin{tabular}{lc}
    \toprule
    Model & $x^*$ \\ \midrule
    Linear & $(-200, 1000)$ \\
    Cubic & $(0.5, -20, 300, 1000)$ \\
    Exponential & $(5000, 4000, 0.2)$ \\
    Logistic & $(6000, -5000, -0.2, -3.7)$ \\
    \bottomrule
  \end{tabular}
  \caption{Exact solutions used for each model in order to generate
    random instances.}
  \label{tab:exactsol}
\end{table}

The parallel version of {\raff} was run with its default parameters on
a Intel Xeon E3-1220 v3 3.10GHz with 4 cores and 16GB of RAM and Linux
LUbuntu 18.04 operating system. The obtained results are displayed in
Tables~\ref{tab:out1} and~\ref{tab:out2}. In those tables, $r$ is the
number of points representing the experiments, $p$ is the number of
trusted points, \texttt{FR} is the ratio of problems in which all the
outliers have been found (but other points may be declared as
outliers), \texttt{ER} is the ratio of problems where exactly the
$r - p$ outliers have been found, \texttt{TP} is the average number of
correctly identified outliers, \texttt{FP} is the average number of
incorrectly identified outliers, \texttt{Avg.}  is the average number
of points that have been declared as outliers by the algorithm and
\texttt{Time} is the total CPU time in seconds to run all the 1000
tests, measured with the \verb+@elapsed+ Julia macro. By default,
$p_{min} = 0.5 r$ and $p_{max} = r$ are set in the algorithm. The
success criteria~(\ref{impl:success}) of \lmlovo\ was set to
$\varepsilon = 10^{-4}$, while $\overline{\lambda}$ was set to
$2$. For each combination (Model, $r$, $p$) there are 4 rows in
Tables~\ref{tab:out1} and~\ref{tab:out2}, representing different
numbers of multistart trials: 1, 10, 100 and 1000.

\begin{table}[h]\centering

\begin{tabular}{l|c|c|rrrrrrrrr}
  \toprule
  Type & $r$ & $p$ & \texttt{FR} & \texttt{ER} & \texttt{TP} & \texttt{FP} & \texttt{Avg.} & Time (s) \\
  \midrule
    \multirow{16}{*}{Linear}
       & \multirow{8}{*}{10}
              &    \multirow{4}{*}{9}
              & 0.858 & 0.552 & 0.858 &  0.349 &  1.21 &    2.167 \\
          & & & 0.859 & 0.554 & 0.859 &  0.347 &  1.21 &    3.511 \\
          & & & 0.859 & 0.554 & 0.859 &  0.347 &  1.21 &   12.893 \\
          & & & 0.859 & 0.554 & 0.859 &  0.347 &  1.21 &   89.285 \\ \cmidrule{3-9} & & \multirow{4}{*}{8}
              & 0.467 & 0.418 & 1.112 &  0.144 &  1.26 &    2.344 \\
          & & & 0.467 & 0.417 & 1.112 &  0.145 &  1.26 &    3.596 \\
          & & & 0.467 & 0.417 & 1.112 &  0.145 &  1.26 &   13.164 \\
          & & & 0.467 & 0.417 & 1.112 &  0.145 &  1.26 &   88.684 \\ \cmidrule{2-9} & \multirow{8}{*}{100} & \multirow{4}{*}{99}
              & 0.983 & 0.078 & 0.983 & 10.656 & 11.64 &   10.297 \\
          & & & 0.982 & 0.074 & 0.982 & 10.655 & 11.64 &   42.091 \\
          & & & 0.982 & 0.074 & 0.982 & 10.677 & 11.66 &  316.604 \\
          & & & 0.982 & 0.075 & 0.982 & 10.682 & 11.66 & 3082.385 \\ \cmidrule{3-9} & & \multirow{4}{*}{90}
              & 0.916 & 0.069 & 9.858 &  6.768 & 16.63 &    9.799 \\
          & & & 0.916 & 0.070 & 9.858 &  6.798 & 16.66 &   40.581 \\
          & & & 0.915 & 0.070 & 9.854 &  6.782 & 16.64 &  317.722 \\
          & & & 0.917 & 0.070 & 9.860 &  6.799 & 16.66 & 3099.536 \\ \midrule

    \multirow{16}{*}{Cubic}
       & \multirow{8}{*}{10}
              &    \multirow{4}{*}{9}
              & 0.767 & 0.572 & 0.767 &  0.290 &  1.06 &    3.062 \\
          & & & 0.810 & 0.563 & 0.810 &  0.371 &  1.18 &    4.111 \\
          & & & 0.886 & 0.549 & 0.886 &  0.461 &  1.35 &   16.370 \\
          & & & 0.886 & 0.545 & 0.886 &  0.465 &  1.35 &  126.554 \\ \cmidrule{3-9} & & \multirow{4}{*}{8}
              & 0.150 & 0.122 & 0.581 &  0.243 &  0.82 &    2.353 \\
          & & & 0.333 & 0.282 & 0.894 &  0.202 &  1.10 &    4.221 \\
          & & & 0.525 & 0.462 & 1.220 &  0.143 &  1.36 &   16.482 \\
          & & & 0.533 & 0.469 & 1.232 &  0.142 &  1.37 &  126.088 \\ \cmidrule{2-9} & \multirow{8}{*}{100} & \multirow{4}{*}{99}
              & 0.990 & 0.046 & 0.990 & 10.997 & 11.99 &   11.485 \\
          & & & 0.991 & 0.041 & 0.991 & 11.351 & 12.34 &   51.548 \\
          & & & 0.992 & 0.037 & 0.992 & 11.788 & 12.78 &  420.033 \\
          & & & 0.993 & 0.036 & 0.993 & 11.706 & 12.70 & 4123.040 \\ \cmidrule{3-9} & & \multirow{4}{*}{90}
              & 0.945 & 0.064 & 9.838 &  6.941 & 16.78 &   11.325 \\
          & & & 0.930 & 0.063 & 9.816 &  7.299 & 17.11 &   50.685 \\
          & & & 0.941 & 0.063 & 9.835 &  7.584 & 17.42 &  414.084 \\
          & & & 0.940 & 0.060 & 9.833 &  7.714 & 17.55 & 4042.454 \\

  \bottomrule
\end{tabular}

\caption{Results of \raff\ for the detection of outliers for
  linear and cubic models. For each kind of problem, a
    multistart strategy was tested with 1, 10, 100 and 1000 random
    starting points.}
\label{tab:out1}
\end{table}

\begin{table}[h]\centering

\begin{tabular}{l|c|c|rrrrrr}
  \toprule
  Type & $r$ & $p$ & \texttt{FR} & \texttt{ER} & \texttt{TP} & \texttt{FP} & \texttt{Avg.} & Time (s) \\
  \midrule
    \multirow{16}{*}{Exponential}
       & \multirow{8}{*}{10}
              &    \multirow{4}{*}{9}
              & 0.549 & 0.141 & 0.549 &  0.751 &  1.30 &     5.627 \\
          & & & 0.698 & 0.463 & 0.698 &  0.354 &  1.05 &    17.289 \\
          & & & 0.777 & 0.535 & 0.777 &  0.338 &  1.11 &   136.491 \\
          & & & 0.822 & 0.581 & 0.822 &  0.306 &  1.13 &  1215.738 \\ \cmidrule{3-9} & & \multirow{4}{*}{8}
              & 0.213 & 0.080 & 0.771 &  0.459 &  1.23 &     4.417 \\
          & & & 0.292 & 0.264 & 0.921 &  0.152 &  1.07 &    18.053 \\
          & & & 0.406 & 0.367 & 1.138 &  0.148 &  1.29 &   138.862 \\
          & & & 0.516 & 0.480 & 1.246 &  0.092 &  1.34 &  1245.882 \\ \cmidrule{2-9} & \multirow{8}{*}{100} & \multirow{4}{*}{99}
              & 0.982 & 0.089 & 0.982 &  5.444 &  6.43 &    47.235 \\
          & & & 0.992 & 0.046 & 0.992 & 10.673 & 11.66 &   392.884 \\
          & & & 0.992 & 0.047 & 0.992 & 10.794 & 11.79 &  3521.630 \\
          & & & 0.993 & 0.044 & 0.993 & 11.298 & 12.29 & 35418.130 \\ \cmidrule{3-9} & & \multirow{4}{*}{90}
              & 0.532 & 0.133 & 8.234 &  1.921 & 10.15 &    47.915 \\
          & & & 0.972 & 0.060 & 9.946 &  7.181 & 17.13 &   384.544 \\
          & & & 0.980 & 0.054 & 9.959 &  7.611 & 17.57 &  3389.777 \\  
          & & & 0.980 & 0.063 & 9.939 &  7.772 & 17.71 & 34121.028 \\ \midrule

  \multirow{16}{*}{Logistic}
         & \multirow{8}{*}{10}
              &    \multirow{4}{*}{9}
              & 0.009 & 0.001 & 0.009 &  0.116 &  0.13 &     2.705 \\
          & & & 0.245 & 0.156 & 0.245 &  0.419 &  0.66 &     3.714 \\
          & & & 0.420 & 0.309 & 0.420 &  0.292 &  0.71 &    18.144 \\
          & & & 0.524 & 0.364 & 0.524 &  0.279 &  0.80 &   150.310 \\ \cmidrule{3-9} & & \multirow{4}{*}{8}
              & 0.003 & 0.001 & 0.091 &  0.400 &  0.49 &     1.914 \\
          & & & 0.032 & 0.028 & 0.396 &  0.369 &  0.77 &     3.932 \\
          & & & 0.065 & 0.059 & 0.389 &  0.285 &  0.67 &    21.091 \\
          & & & 0.167 & 0.143 & 0.559 &  0.203 &  0.76 &   175.915 \\ \cmidrule{2-9} & \multirow{8}{*}{100} & \multirow{4}{*}{99}
              & 0.535 & 0.006 & 0.535 &  7.105 &  7.64 &     9.426 \\
          & & & 0.536 & 0.012 & 0.536 & 11.754 & 12.29 &    34.022 \\
          & & & 0.894 & 0.063 & 0.894 &  2.529 &  3.42 &   309.246 \\
          & & & 0.929 & 0.095 & 0.929 &  5.276 &  6.21 &  2678.522 \\ \cmidrule{3-9} & & \multirow{4}{*}{90}
              & 0.002 & 0.000 & 4.345 &  3.295 &  7.64 &     9.459 \\
          & & & 0.008 & 0.001 & 4.599 &  5.502 & 10.10 &    38.605 \\
          & & & 0.432 & 0.001 & 6.551 &  4.629 & 11.18 &   319.713 \\
          & & & 0.430 & 0.099 & 8.084 &  2.001 & 10.09 &  2774.727 \\
  
  \bottomrule
\end{tabular}

\caption{Results for \raff\ for the detection of outliers for
  exponential and logistic models. For each kind of problem, a
    multistart strategy was tested with 1, 10, 100 and 1000 random
    starting points.}
\label{tab:out2}
\end{table}

Some conclusions can be drawn from Tables~\ref{tab:out1}
and~\ref{tab:out2}. We can see that \raff\ attains its best
performance for outlier detection when the number of correct points is
not small, even though the percentage of outliers is high. For the
exponential and logistic models, we also can see clearly the effect of
the multistart strategy in increasing the ratio of identified
outliers. In problems with 100 experiments, we observe that in almost
all the cases the number of outliers have been overestimated in
average: although the ratio of outlier identification is high
(\texttt{FR}), the ratio of runs where only the exact outliers have
been detected (\texttt{TR}) is very low, being below 20\% of the
runs. For small test sets, this ratio increases up to 50\%, but
difficult models, such as the exponential and logistic, have very low
ratios. However, as we can observe in Figure~\ref{fig:out}, the shape
and the parameters of the model are clearly free from the influence of
outliers. This observation suggests that maybe the perturbation added
to all the values is causing the algorithm to detect correct points as
outliers. The effect of the number of multi-start runs linearly
increases the runtime of the algorithm, but is able to improve the
adjustment, specially for the logistic model. The exponential model
has an awkward behavior, where the \texttt{ER} ratio decreases when
the number of multi-start runs increases, although the ratio of
problems where all the outliers have been detected increases
(\texttt{FR}). This might indicate that the tolerance~(\ref{impl:tvs})
could be improved. We can also observe that the runtime of the
exponential model is ten times higher than the other models.

When the size of the problem is multiplied by 10 (from 10 points to
100), we observe that the CPU time is multiplied by 5. This occurs
because the time used by communication in the parallel runs is less
important for larger datasets. Again, the exponential model is an
exception.

\begin{figure}[ht]
  \centering

  \begin{tabular}{cc}
    \includegraphics[width=0.45\textwidth]{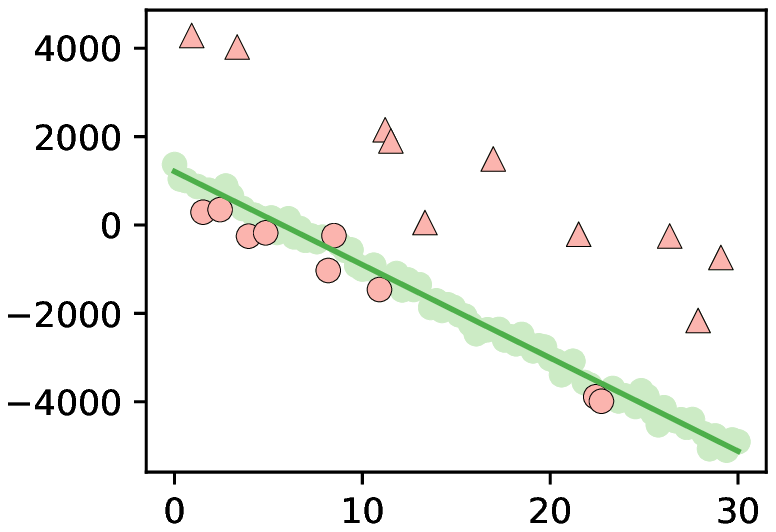} &
    \includegraphics[width=0.45\textwidth]{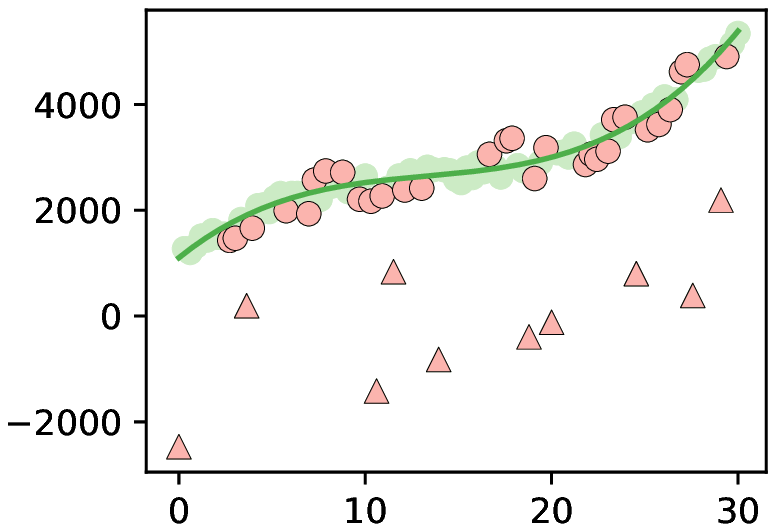} \\
    Linear & Cubic \\
    \includegraphics[width=0.45\textwidth]{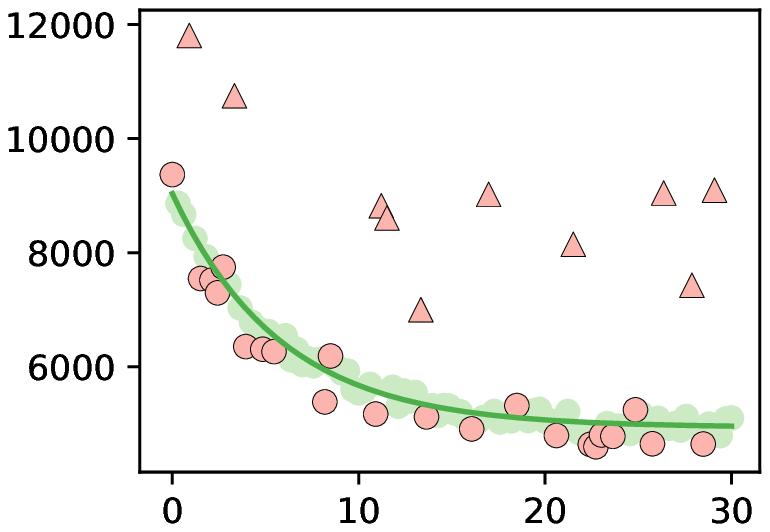} &
    \includegraphics[width=0.45\textwidth]{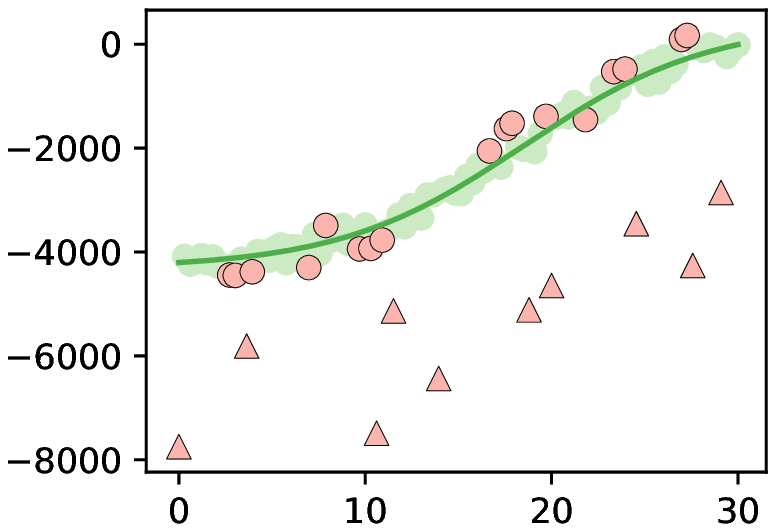} \\
     Exponential & Logistic \\
  \end{tabular}

  \caption{Selected instances of test problems with $r = 100$
      and $p = 90$ and the solutions obtained by {\raff}. All the
    outliers have been correctly identified in those cases
    (dark/red triangles). Non-outliers are described by circles, where
    the dark/red ones represent points incorrectly classified as
    outliers by the algorithm.}
\label{fig:out}
\end{figure}

In a second round of experiments, the same procedure was used to
generate random test problems simulating results from 100 experiments
($r = 100$), where a cluster of 10\% of the points are outliers
($p = 90$). The default interval used for the values of $t$
is $[1, 30]$, and the clustered outliers always belong to $[5,
10]$. Selected instances for each type of model are shown in
Figure~\ref{fig:cout} as well as the solution found by {\raff}. Again,
1000 random problems were generated for each type of model and the
multi-start procedure was fixed to 100 starting points for each
problem. The
obtained results are shown in Table~\ref{tab:cout}. A clustered set of
outliers can strongly affect the model but is also easier to detect,
when the number of outliers is not very large. As we can observe in
Table~\ref{tab:cout}, the ratio of instances where all the outliers
have been successfully detected has increased in all models. The
logistic model is the most difficult to fit since, on average, \raff\
detects 17 points as outliers and 9 of them are
correctly classified (\texttt{TP}). All the other models are able to
correctly identify 10 outliers, on average, and have a higher
\texttt{FR} ratio.

This set of experiments also shows another benefit of the present
approach. If the user roughly knows the number of points that belong
to a given model, such information can be used in the elimination of
random (not necessary Gaussian) noise. The clustered example will be
also shown to be an advantage over traditional robust least-squares
algorithms in Subsections~\ref{sec:impl:algs}
and~\ref{sec:impl:circle}.

\begin{figure}[ht]
  \centering

  \begin{tabular}{cc}
    \includegraphics[width=0.45\textwidth]{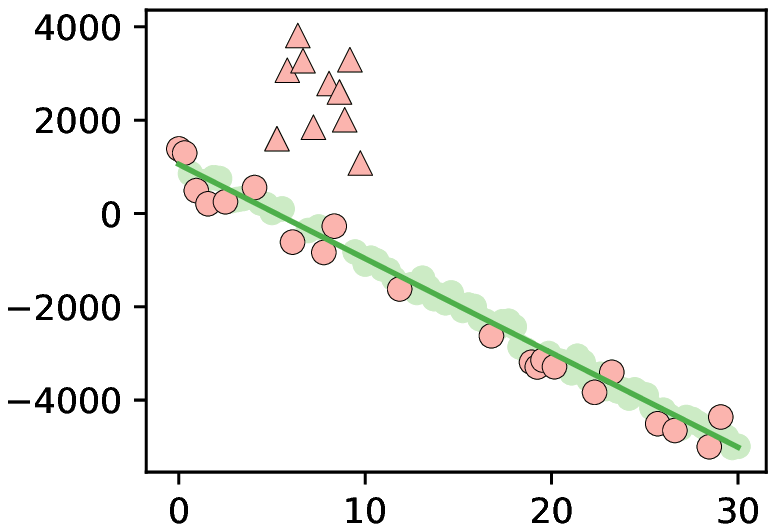}
    & \includegraphics[width=0.45\textwidth]{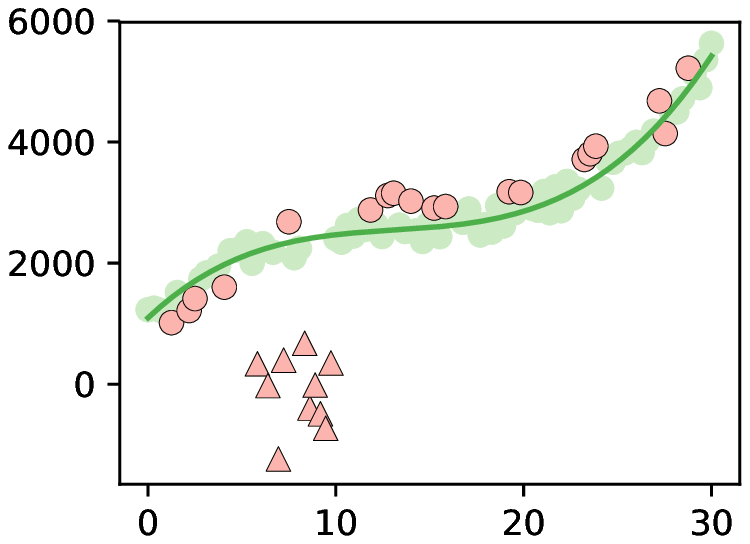}\\
    Linear & Cubic \\
    \includegraphics[width=0.45\textwidth]{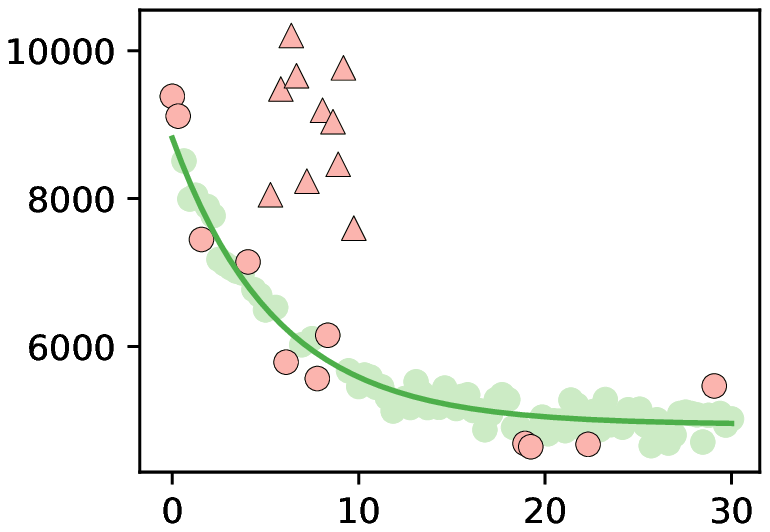}
    & \includegraphics[width=0.45\textwidth]{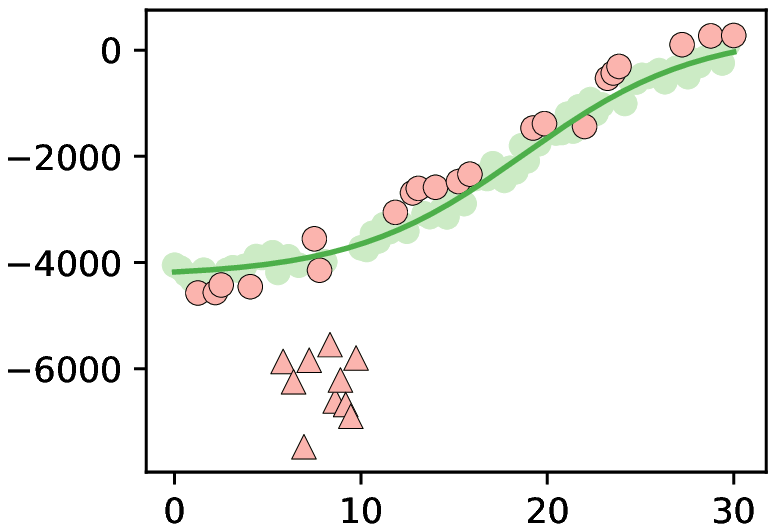}\\
    Exponential & Logistic
  \end{tabular}
  
  \caption{Selected instances of test problems containing a clustered
    set of outliers and the solutions obtained by {\raff}. All the
    outliers have been correctly identified in those cases
    (dark/red triangles). Non-outliers are described by circles, where
    the dark/red ones represent points incorrectly classified as
    outliers by the algorithm.}
\label{fig:cout}
\end{figure}

\begin{table}[h]
  \centering

  \begin{tabular}{lrrrrrrrrr}
    \toprule
    Type & \texttt{FR} & \texttt{ER} & \texttt{TP} & \texttt{FP} & \texttt{Avg.} & Time (s) \\
    \midrule

       Linear & 0.949 & 0.104 & 9.936 & 6.363 & 16.30 &  323.623 \\
        Cubic & 0.991 & 0.047 & 9.991 & 8.368 & 18.36 &  423.634 \\
  Exponential & 0.987 & 0.097 & 9.983 & 6.775 & 16.76 & 3445.482 \\
     Logistic & 0.745 & 0.007 & 8.778 & 8.382 & 17.16 &  326.675 \\
  
   \bottomrule
  \end{tabular}

  \caption{Numerical results for problems with $p = 100$ data points
    and 10\% of clustered outliers.}
  \label{tab:cout}
\end{table}

\subsection{Comparison against robust algorithms}
\label{sec:impl:algs}

We compared the fitting obtained by \raff\ against classical and
robust fitting algorithms provided by the SciPy library version 1.3.1
in
Python\footnote{\url{https://docs.scipy.org/doc/scipy/reference/optimize.html}}. The
robust fitting algorithm in SciPy consists of using different loss
functions in the least squares formulation. The following loss
functions were used: \texttt{linear} (usual least squares
formulation), \texttt{soft\_l1} (smooth approximation of the $\ell_1$
loss function), \texttt{huber} and \texttt{cauchy}. The
\texttt{PyCall.jl} Julia library was used to load and call
  SciPy.

Two more algorithms based on the RANSAC (Random Sample
Consensus)~\cite{Fischler1981}, implemented in C$++$ from the Theia
Vision Library\footnote{\url{http://www.theia-sfm.org/ransac.html}}
version 0.8, were considered. The first one, called here
\texttt{RANSAC}, is the traditional version of RANSAC and the second
one is \texttt{LMED}, based on the work~\cite{Rousseeuw1984}, which
does not need the error threshold, the opposite of case of
\texttt{RANSAC} (where the threshold is problem dependent).

All the algorithms from SciPy were run with their default
parameters. The best model among 100 runs was selected as the solution
for each algorithm and the starting point used was randomly generated
following the normal distribution with $\mu = 0$ and $\sigma = 1$.
Algorithms \texttt{RANSAC} and \lmed\ were run only 10 times, due to
the higher CPU time used and the good quality of the solution
achieved. \texttt{RANSAC} and \lmed\ were run with a maximum of 1000
iterations, sampling 10\% of the data and with the MLE score parameter
activated. In order to adjust models to the sampled data, the Ceres
least squares solver\footnote{http://ceres-solver.org/} version 1.13.0
was used, since Theia has a natural interface to it. All the scripts
used in the tests are available at
\url{https://github.com/fsobral/RAFF.jl}. Once again, the parallel
version of \raff\ was used. The test problems were generated by the
same procedures discussed in
Subsection~\ref{sec:impl:outlier}. However, only one problem (instead
of 1000) for each configuration (Model, $r$, $p$) was used.

Unlike \raff, traditional fitting algorithms do not return the
possible outliers of a dataset. Robust algorithms such as least
squares using $\ell_1$ or Huber loss functions are able to ignore the
effect of outliers, but not to easily detect them. Therefore, for the
tests we selected one instance of each test of type (model, $r$,
$p$), where the models and values for $r$ and $p$ are the same used
in Tables~\ref{tab:out1}--\ref{tab:cout}. The results are displayed
in Table~\ref{tab:comparison}. For each problem $p$ and each algorithm
$a$, we measured the adjustment error $A_{a, p}$ between the model
obtained by the algorithm $\phi_p(x_{a, p}^\star , t)$ and the points
that are non-outliers, which is given by
\[
  A_{a, p} = \sqrt{\sum \limits_{\substack{i \in \mathcal{P}\\ i\
        \mbox{non-outlier}}} (\phi_p(x_{a,p}^\star, t_i) - y_i)^2},
\]
where $\phi_p$ was the model used to adjust problem $p$. Each row of
Table~\ref{tab:comparison} represents one problem and contains the
relative adjustment error for each algorithm, which is
defined by
\begin{equation}
  \label{eq:adjerror}
  \bar A_{a, p} = \frac{A_{a, p}}{\min_{i} \{A_{i, p}\}}
\end{equation}
and the time taken to find the model (in parenthesis). The last row
contains the number of times that each algorithm has found a solution
with adjustment error smaller than 1\% of best, smaller than 10\% of
the best and smaller than 20\% of the best adjustment measure found
for that algorithm, respectively, in all the test set. We can
observe that \raff, \lmed\ and \texttt{soft\_l1} were the best
algorithms. \raff\ was the solver that found the best models in most
of the problems (11/24), followed by \lmed\ (9/24). Its parallel
version was consistently the fastest solver among all. It is important
to observe that \raff\ was the only who was easily adapted to run in
parallel. However, the parallelism is related only to the solution of
subproblems for different $p$, not to the multistart runs,
which are run sequentially. Therefore, \raff\ solves
considerably more problems per core than the other algorithms in a
very competitive CPU time. When parallelism is turned of, the CPU time
is very similar to the traditional least squares algorithm
(\texttt{linear}). Also, \raff\ was the only one that easily outputs
the list of possible outliers without the need of any threshold
parameter. Clustered instance (cubic, 100, 90) and instance (logistic,
10, 9) and the models obtained by each algorithm are shown in
Figure~\ref{fig:comparison}.

\begin{sidewaystable}[ht!]
  \centering
  \begin{tabular}{l|rrrrrrr}
    
    \toprule
    (Model, $r$, $p$)  & \texttt{            linear }& \texttt{           soft\_l1 }& \texttt{             huber }& \texttt{            cauchy }& \texttt{            RANSAC }& \texttt{              LMED }& \texttt{           RAFF } \\ \midrule
      linear,   10,    9 &       1.35 (  0.60)&       1.06 (  2.75)&       1.06 (  3.11)&       1.37 (  1.64)&       4.28 (  0.52)&       1.00 (  0.53)&       1.35 (  1.48) \\
      linear,   10,    8 &       3.62 (  0.60)&       1.18 (  2.55)&       1.18 (  2.89)&       1.14 (  2.07)&       1.17 (  0.54)&       1.01 (  0.55)&       1.00 (  0.12) \\
      linear,  100,   99 &       1.03 (  0.60)&       1.00 (  3.22)&       1.00 (  8.43)&       1.00 (  2.30)&       1.22 (  3.95)&       1.00 (  3.98)&       1.00 (  0.82) \\
      linear,  100,   90 &       1.75 (  0.60)&       1.00 (  3.41)&       1.00 (  8.41)&       1.04 (  1.65)&       1.23 (  4.71)&       1.01 ( 47.23)&       1.04 (  0.67) \\
       cubic,   10,    9 &       1.27 (  0.81)&       1.00 (  9.44)&       1.11 ( 11.42)&       4.71 (  4.67)&      22.26 (  0.98)&       2.47 (  0.99)&       4.91 (  1.04) \\
       cubic,   10,    8 &       4.06 (  0.73)&       1.18 ( 14.80)&       1.18 ( 16.28)&       1.19 (  4.75)&     114.88 (  1.00)&       1.21 (  1.00)&       1.00 (  0.18) \\
       cubic,  100,   99 &       1.05 (  0.72)&       1.00 ( 21.08)&       2.24 ( 22.20)&       1.09 (  6.00)&       1.12 (  8.00)&       1.04 (  8.04)&       1.08 (  0.81) \\
       cubic,  100,   90 &       1.76 (  0.73)&       1.00 ( 21.59)&       1.78 ( 22.86)&       1.19 (  5.79)&       1.15 (  7.94)&       1.00 ( 79.15)&       1.00 (  0.75) \\
       expon,   10,    9 &       1.16 (  5.04)&      12.88 ( 12.26)&      12.88 ( 12.59)&       4.68 ( 10.70)&      14.22 (  0.76)&       1.00 (  0.76)&       1.16 (  1.35) \\
       expon,   10,    8 &       1.34 (  3.28)&       7.74 ( 12.14)&      10.07 ( 12.45)&       1.00 ( 10.89)&       1.72 (  0.72)&       1.61 (  0.73)&       1.34 (  0.25) \\
       expon,  100,   99 &       1.00 (  3.86)&       8.57 ( 13.61)&       8.58 ( 12.86)&       1.22 ( 10.99)&       2.04 (  5.32)&       1.05 (  5.32)&       1.03 (  5.19) \\
       expon,  100,   90 &       1.76 (  3.65)&       8.99 ( 13.53)&       8.99 ( 13.18)&       3.70 ( 11.04)&       1.93 (  5.35)&       1.00 ( 53.63)&       1.02 (  5.32) \\
    logistic,   10,    9 &       1.68 (  3.24)&       1.31 ( 18.98)&       7.60 ( 20.55)&       1.91 ( 17.28)&      24.45 (  0.89)&       2.27 (  0.90)&       1.00 (  0.99) \\
    logistic,   10,    8 &       4.23 ( 10.27)&       1.00 ( 18.14)&       9.40 ( 19.18)&      26.13 (  3.98)&      39.56 (  0.85)&       1.05 (  0.87)&      22.19 (  0.14) \\
    logistic,  100,   99 &       1.01 (  4.44)&       2.34 ( 17.68)&       8.88 ( 18.84)&       2.64 (  8.24)&       1.02 (  7.75)&       1.00 (  7.76)&       1.01 (  0.51) \\
    logistic,  100,   90 &       1.78 (  2.57)&       1.10 ( 18.16)&       7.13 ( 19.10)&       7.53 (  8.06)&       1.06 (  7.59)&       1.00 ( 76.37)&       1.01 (  0.53) \\ \midrule \multicolumn{8}{c}{Clustered} \\ \midrule
      linear,   10,    8 &       4.10 (  0.52)&       1.18 (  2.33)&       1.18 (  2.54)&       1.00 (  2.02)&       4.69 (  0.54)&       1.07 (  0.55)&       1.00 (  0.10) \\
      linear,  100,   90 &       1.92 (  0.62)&       1.00 (  2.76)&       1.00 (  5.70)&       1.12 (  2.10)&       1.11 (  4.66)&       1.02 ( 46.45)&       1.03 (  0.68) \\
       cubic,   10,    8 &       4.29 (  0.72)&       1.00 ( 11.90)&       1.00 ( 14.59)&      11.59 (  4.87)&      31.72 (  1.04)&      10.79 (  1.02)&      12.61 (  0.14) \\
       cubic,  100,   90 &       2.18 (  0.73)&       1.02 ( 20.39)&       2.47 ( 22.25)&       1.07 (  6.31)&       1.52 (  7.70)&       1.09 ( 73.61)&       1.00 (  0.83) \\
       expon,   10,    8 &       1.32 ( 12.79)&       4.11 ( 10.54)&       1.81 ( 10.34)&       1.20 (  9.18)&       1.00 (  0.75)&       2.69 (  0.78)&       4.22 (  0.29) \\
       expon,  100,   90 &       1.99 (  3.56)&       8.58 ( 13.31)&       8.57 ( 12.79)&       4.99 ( 10.90)&       1.74 (  5.66)&       1.14 ( 57.14)&       1.00 (  5.82) \\
    logistic,   10,    8 &       3.79 ( 10.17)&       1.00 ( 15.28)&       7.74 ( 18.09)&      22.46 ( 10.53)&      16.32 (  0.66)&       1.40 (  0.65)&      12.31 (  0.13) \\
    logistic,  100,   90 &       2.00 (  4.87)&       4.62 ( 18.10)&       6.77 ( 19.83)&       7.35 (  7.10)&       1.40 (  7.78)&       1.03 ( 78.43)&       1.00 (  0.52) \\ \midrule
                         &     2,     5,     6&     9,    11,    15&     4,     5,     9&     3,     6,    11&     1,     3,     7&     9,    16,    17&    11,    16,    17 \\ \bottomrule

  \end{tabular}

  \caption{Comparison against different robust fitting algorithms
    using one instance of each test problem and 100 random
      starting points as a multistart strategy. \texttt{RANSAC} and \lmed\
      run for 10 random initial points.}
  \label{tab:comparison}
\end{sidewaystable}

\begin{figure}[ht!]
  \centering
  \begin{tabular}{cc}
    \includegraphics[width=0.45\textwidth]{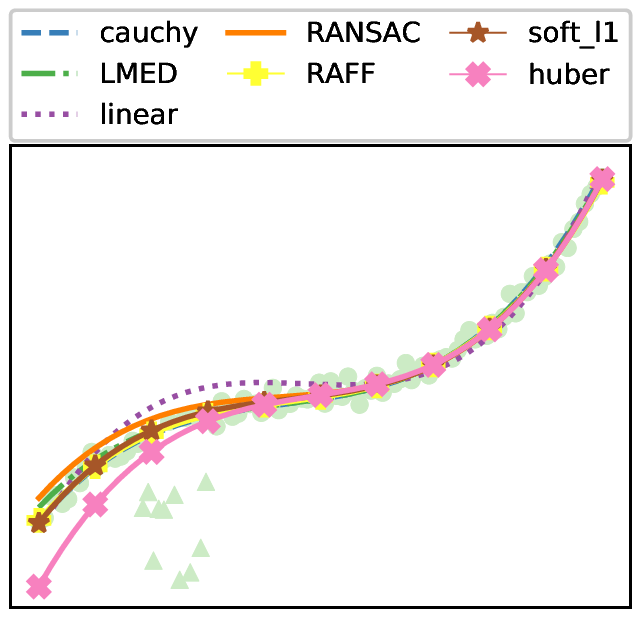}
    & \includegraphics[width=0.45\textwidth]{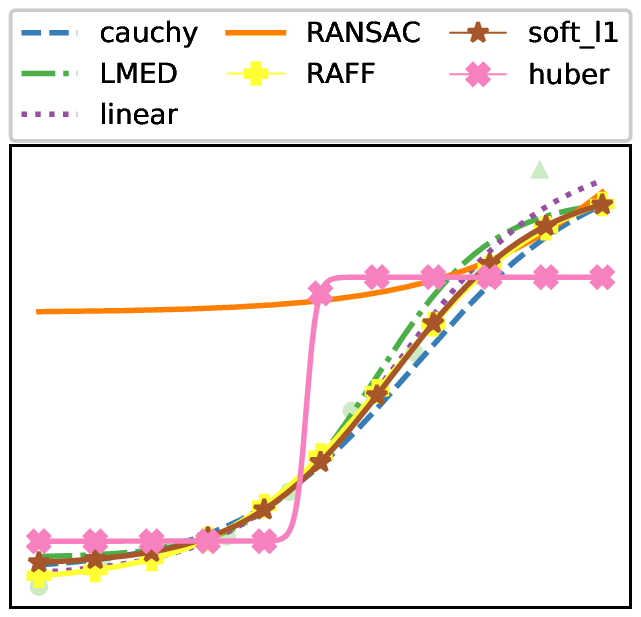}
  \end{tabular}
  \caption{Two problems and the models found for each algorithm. On
    the left, a cubic model with 100 points and a set of 10 clustered
    outliers. On the right, a logistic model with 10 points and only
    one outlier.}
  \label{fig:comparison}
\end{figure}

\subsection{Experiments for circle detection}
\label{sec:impl:circle}

The problem of detecting patterns in images is very discussed in the
vision area in Computer Science. LOVO algorithms have also been
applied to solving such problems, as a nonlinear programming
alternative to traditional
techniques~\cite{Andreani2007curvedect}. The drawback, again, is the
necessity of providing a reasonable number of trusted points. {\raff}
allows the user to provide an interval of possible trusted points, so
the algorithm can efficiently save computational effort when trying to
find patterns in images. Since LOVO problems need a model to be
provided, circle detection is a perfect application to the algorithm.

We followed tests similar to~\cite{Yu2010a}, using a circular model
\[
  \phi(x, t) = (t_1 - x_1)^2 + (t_2 - x_2)^2 - x_3^2
\]
instead of the ellipse model considered in the work. Two test sets
were generated. In the first set $r = 100$ points were uniformly
distributed in the border of the circle with center $(-10, 30)$ and
radius $2$. If the point is not an outlier, a random perturbation
$\xi \sim \Norm(0, 0.1)$ is added to each one of its $t_1$ and $t_2$
coordinates. For outliers, the random noise is given by
$\xi \sim \Norm(0, 2)$, as suggested in~\cite{Yu2010a}. In the second
set, $r = 300$ was considered. The same circumference was used and $p$
points (non-outliers) were uniformly distributed in the circumference
and slightly perturbed with a noise $\xi \sim \Norm(0, 0.1)$ as
before. The remaining $300 - p$ points (the outliers) were randomly
distributed in a square whose side was 4 times the radius of the
circle, using the uniform distribution. Nine problems were generated
in each test set, with outlier ratio ranging from 10\% up to 90\%
(i. e. ratio of non-outliers decreasing 90\% to 10\%).

The same algorithms were compared, the only difference from
Subsection~\ref{sec:impl:algs} is that the error threshold of
\texttt{RANSAC} was reduced to 10 and 100 random starting points near
$(1, 1, 1)$ were used for all algorithms, except \texttt{RANSAC} and
\lmed. For those two algorithms, we kept the number of trials
to 10. Also, we tested two versions of \raff. In pure \raff, we
decreased the lower bound $p_{min}$ from its default value $0.5 r$ to
the value of $p$, when $p$ falls below the default. In \raff int, we
used the option of providing upper and lower bounds for the number of
trusted points. If $p$ is the current number of non-outliers in the
instance, the interval given to \raff int is
$[p - 0.3r, p + 0.3r] \cap [0, r]$. The measure~(\ref{eq:adjerror})
was used and the results are shown in Figure~\ref{fig:circleerr}.

We can observe that \raff, \lmed\ and \texttt{cauchy} achieved the
best results. \raff\ found worse models than most of the robust
algorithms in the problems of the first test set, although the results
are still very close. Its relative performance increases as the
outlier ratio increases. This can be explained as the strong
attraction that \raff\ has to finding a solution similar to
traditional least squares algorithms. In Figure~\ref{fig:interval} we
can see that \raff\ has difficulty in finding outliers that belong to
the interior of the circle. To solve this drawback, \raff\ also
accepts a lower bound in the number of outliers, rather than only an
upper bound. This ability is useful for large datasets with a lot of
noise, as is the case of the second test set, and allows the detection
of inner outliers. This is represented by \raff int. We can see in
Figure~\ref{fig:circleerr} that the performance of both versions of
\raff\ is better than traditional robust algorithms in the case of a
large number of outliers.

\begin{figure}[ht]
  \centering
  \includegraphics[width=0.9\textwidth]{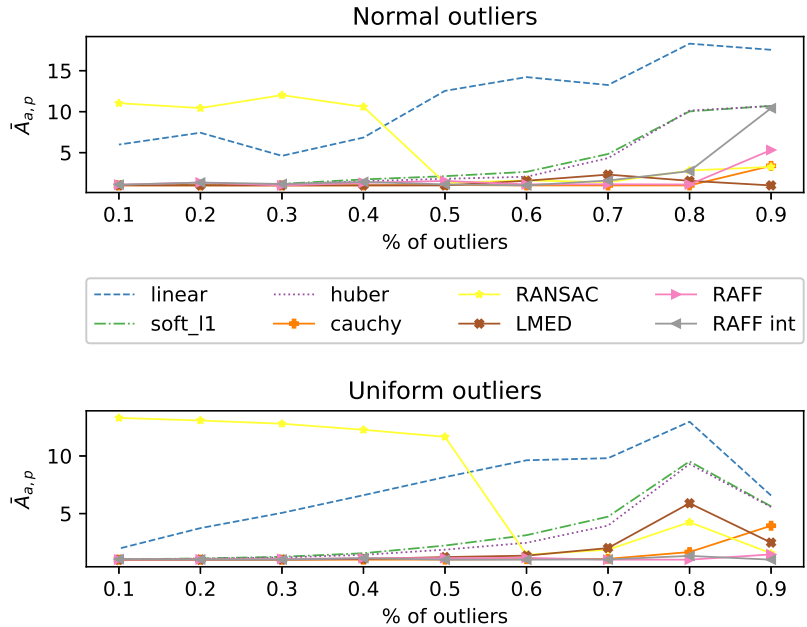}
  \caption{Relative adjustment error in the circle detection problem
    for increasing outlier ratio and two different types of
      perturbation: by normal and uniform
    distributions.}
  \label{fig:circleerr}
\end{figure}

\begin{figure}[ht!]
  \centering

  \begin{tabular}{cc}
    \toprule
    \multicolumn{2}{c}{Outliers generated by normal distribution} \\ \midrule
    \includegraphics{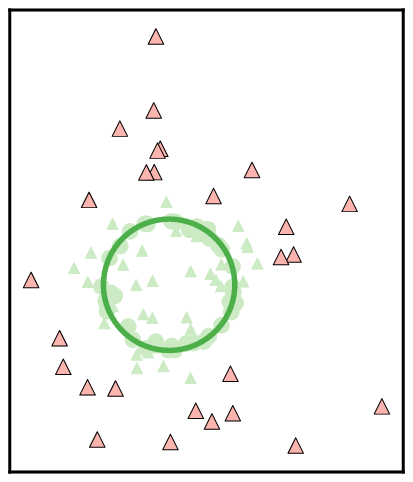}
    & \includegraphics{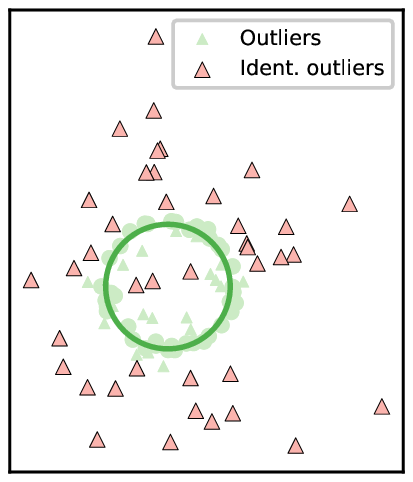} \\
    At most 40\% of outliers & Outlier ratio between 10\% and 60\% \\ \midrule
    \multicolumn{2}{c}{Outliers generated by uniform distribution}\\ \midrule
    \includegraphics{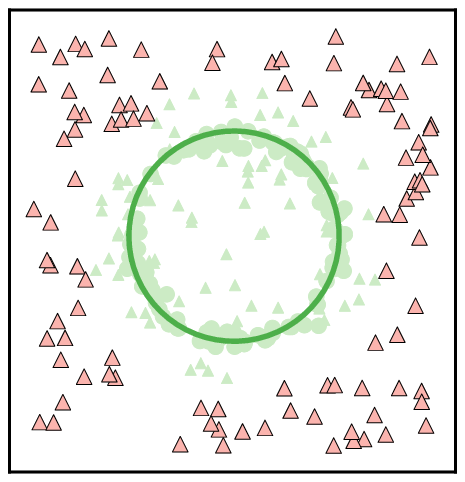}
    & \includegraphics{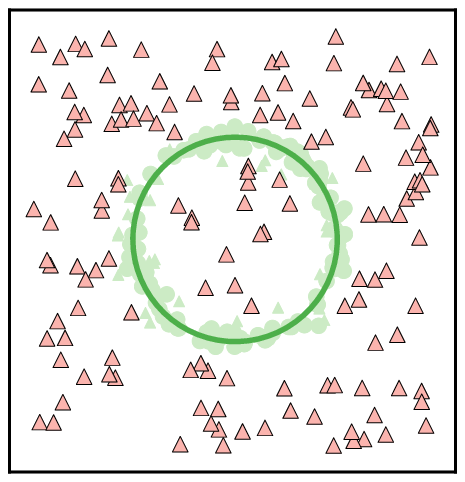} \\
    At most 40\% of outliers & Outlier ratio between 20\% and 60\% \\ \bottomrule
  \end{tabular}
  
  \caption{Difference of outlier detection when upper bounds on the
    outlier ratio are provided. The inner outliers are harder to
    detect, since the error they cause in the model is smaller.}
  \label{fig:interval}
\end{figure}

\section{Conclusions}
\label{sec:conclusions}

In this paper, we have described a LOVO version of the
Levenberg-Marquardt algorithm for solving nonlinear equations, which
is specialized to the adjustment of models where the data contains
outliers. The theoretical properties of the algorithm were studied and
convergence to weakly stationary points has been proved. To overcome
the necessity of providing the number of outliers in the algorithm, a
voting system has been proposed. A complete framework to robust
adjustment of data was implemented in the Julia language and compared
to public available and well tested robust fitting algorithms. The
proposed algorithm was shown to be competitive, being able to find
better adjusted models in the presence of outliers in most of the
problems. In the circle detection problem, the proposed algorithm was
also shown to be competitive and had a good performance even when the
outlier ration exceeds 50\%. The implemented algorithm and all the
scripts used for testing and generation of the tests are freely
available and constantly updated at
\url{https://github.com/fsobral/RAFF.jl}.

\bibliography{library}

\end{document}